\documentclass[portuges,12pt,letter]{article}
\usepackage[centertags]{amsmath}
\usepackage{amsfonts}
\usepackage{newlfont}
\usepackage{amscd}
\usepackage{graphics}
\usepackage{epsfig}
\usepackage{indentfirst}
\usepackage{amsxtra}
\usepackage[latin1]{inputenc}
\usepackage{amssymb, amsmath}
\usepackage{amsthm}
\usepackage[mathscr]{eucal}

\newtheorem{thm}{Theorem}[section]

\author{Fabio Silva Botelho and Eduardo Pandini Barros \\ Departamento de Matemática \\ Universidade Federal de Santa Catarina, UFSC \\
Florian\'{o}polis, SC - Brazil}

\title{\bf  Existence of solution for  an optimal control problem associated to the Ginzburg-Landau system in superconductivity} 

\begin{document}
\maketitle

\abstract{This article develops a global existence result for the solution of an optimal control problem associated to the Ginzburg-Landau
system. This main result is based on standard tools of analysis and functional analysis, such as the Friedrichs Curl Inequality and the Rellich-Kondrashov Theorem. In the concerning model, we consider the presence of an external magnetic field and the control variable is a complex function acting on the super-conducting sample boundary. Finally the state variables are the
 Ginzburg-Landau order parameter and the magnetic potential, defined on domains properly specified.}

\section{Introduction} This work develops an existence result for an optimal control problem closely related to the Ginzburg-Landau system
in superconductivity.
First, we recall that about the year 1950 Ginzburg and Landau introduced a theory to model the super-conducting behavior of some types of materials below a critical temperature $T_c$,
which depends on the material in question. They postulated the free density energy may be written close to $T_c$ as
$$F_s(T)=F_n(T)+\frac{\hbar}{4m}\int_\Omega |\nabla \psi|^2_2 \;dx+\frac{\alpha(T)}{4}\int_\Omega |\psi|^4\;dx-\frac{\beta(T)}{2}\int_\Omega |\psi|^2\;dx,$$
where $\psi$ is a complex parameter, $F_n(T)$ and $F_s(T)$ are the normal and super-conducting free energy densities, respectively (see \cite{100}
 for details).
Here $\Omega \subset \mathbb{R}^3$ denotes the super-conducting sample with a boundary denoted by $\partial \Omega=\Gamma.$ The complex function $\psi \in W^{1,2}(\Omega; \mathbb{C})$ is intended to minimize
$F_s(T)$ for a fixed temperature $T$.

Denoting $\alpha(T)$ and $\beta(T)$ simply by $\alpha$ and $\beta$,  the corresponding Euler-Lagrange equations are given by:
\begin{equation} \left\{
\begin{array}{ll}
 -\frac{\hbar}{2m}\nabla^2 \psi+\alpha|\psi|^2\psi-\beta\psi=0, & \text{ in } \Omega
 \\ \\
 \frac{\partial {\psi}}{\partial \textbf{n}}=0, &\text{ on } \partial\Omega.\end{array} \right.\end{equation}
This last system of equations is well known as the Ginzburg-Landau (G-L) one.
In the physics literature is also well known the G-L energy in which a magnetic potential here denoted by $\textbf{A}$ is included.
The functional in question is given by:
\begin{eqnarray}\label{ar67}
J(\psi,\textbf{A})&=&\frac{1}{8\pi}\int_{\mathbb{R}^3} |\text{ curl }\textbf{A}-\textbf{B}_0|_2^2\;dx+\frac{\hbar^2}{4m}\int_\Omega \left|\nabla \psi-\frac{2 ie}{\hbar c}\textbf{A}\psi\right|^2_2\;dx \nonumber \\ &&+\frac{\alpha}{4}\int_\Omega|\psi|^4\;dx-\frac{\beta}{2}\int_\Omega |\psi|^2\;dx
\end{eqnarray}
Considering its minimization on the space $U$, where $$U= W^{1,2}(\Omega; \mathbb{C}) \times W^{1,2}(\mathbb{R}^3; \mathbb{R}^3),$$ through the physics notation the corresponding Euler-Lagrange equations are:
\begin{equation} \left\{
\begin{array}{ll}
 \frac{1}{2m}\left(-i\hbar\nabla -\frac{2e}{c}\mathbf{A}\right)^2 \psi+\alpha|\psi|^2\psi-\beta\psi=0, & \text{ in } \Omega
 \\ \\
 \left(i\hbar \nabla\psi+\frac{2e}{c}\textbf{A}\psi\right) \cdot \textbf{n}=0, &\text{ on } \partial\Omega,\end{array} \right.\end{equation}
and
\begin{equation} \left\{
\begin{array}{ll}
 \text{curl }(\text{curl } \textbf{A})= \text{ curl } \textbf{B}_0+\frac{4 \pi}{c} \tilde{J}, & \text{ in } \Omega
 \\ \\
 \text{curl }(\text{curl } \textbf{A})=\text{ curl }\textbf{B}_0, & \text{ in } \mathbb{R}^3\setminus \overline{\Omega},\end{array} \right.\end{equation}
 where $$\tilde{J}=-\frac{ie \hbar}{2m}\left(\psi^*\nabla \psi-\psi\nabla \psi^*\right)-\frac{2e^2}{mc}|\psi|^2 \textbf{A}.$$
 and $$\textbf{B}_0 \in L^2(\mathbb{R}^3; \mathbb{R}^3)$$ is a known applied magnetic field.

Existence of a global solution for a similar problem has been proved in \cite{500}.

\section{An existence result for a related optimal control problem}

Let $\Omega \subset \mathbb{R}^3$, $\Omega_1 \subset \mathbb{R}^3$ be open, bounded and connected sets  with Lipschitzian boundaries, where $\overline{\Omega} \subset \Omega_1$ and $\Omega_1$ is convex. Let $\phi_d:\Omega \rightarrow \mathbb{C}$   be a known function in $L^4(\Omega;\mathbb{C})$ and consider the problem of minimizing
$$\||\phi|^2-|\phi_d|^2\|^2_{0,2,\Omega}$$ with $(\phi,\mathbf{A},u)$ subject to the satisfaction of the Ginzburg-Landau equations, indicated in $(\ref{a.1})$
and $(\ref{a.2})$ in the next lines.

For such a problem, the control variable is $u \in L^2(\partial \Omega;\mathbb{C})$ and the state variables are the Ginzburg-Landau order parameter
$\phi \in W^{1,2}(\Omega,\mathbb{C})$  and the magnetic potential $\mathbf{A} \in W^{1,2}(\Omega_1,\mathbb{R}^3).$

Our main existence result is summarized by the following theorem.

\begin{thm} Consider the functional $$J(\phi,\mathbf{A},u)=\frac{\varepsilon}{2} \|\nabla \phi\|_{0,2,\Omega}^2
+K_1\||\phi|^2-|\phi_d|^2\|_{0,2,\Omega}^2+K_2\|u\|_{0,2,\partial \Omega}^2,$$
subject to
$(\phi,\mathbf{A},u) \in \mathcal{C},$ where
$$\mathcal{C}=\{(\phi,\mathbf{A},u) \in W^{1,2}(\Omega,\mathbb{C}) \times W^{1,2}(\Omega_1,\mathbb{R}^3)\times L^2(\partial \Omega;\mathbb{C})\;:\; \text{ such that }
(\ref{a.1}) \text{ and } (\ref{a.2}) \text{ hold }\},$$ where
\begin{equation}\label{a.1}\left\{
\begin{array}{ll}
 \frac{1}{2m}\left(-i\hbar \nabla -\frac{2e}{c}\mathbf{A} \right)^2\phi+\alpha |\phi|^2\phi -\beta \phi =0,& \text{ in } \Omega,
  \\ \\
  \left(i\hbar \nabla \phi+ \frac{2e}{c}\mathbf{A}\phi\right) \cdot \mathbf{n}=u,& \text{ on } \partial  \Omega,\end{array} \right.\end{equation}
and
\begin{equation}\label{a.2}\left\{
\begin{array}{ll}
 \text{ curl } \text{ curl } \mathbf{A} =\text{ curl }\mathbf{B}_0+\frac{4\pi}{c} \tilde{J},& \text{ in } \Omega,
  \\ \\
  \text{ curl } \text{ curl } \mathbf{A}=\text{ curl } \mathbf{B}_0,& \text{ in } \Omega_1\setminus \Omega,
  \\ \\ \text{ div }\mathbf{A}=0,& \text{ in } \Omega_1,
  \\ \\
  \mathbf{A} \cdot \mathbf{n}=0,& \text{ on } \partial \Omega_1
  \end{array} \right.\end{equation}
 where,
 $$\tilde{J}=-\frac{ie\hbar}{2m}(\phi^*\nabla \phi-\phi\nabla \phi^*)-\frac{2e^2}{mc}|\phi|^2 \mathbf{A},$$
 and where $\varepsilon>0$ is a small parameter, $K_1>0$ and $K_2>0$.

 Under such hypotheses, there exists $(\phi_0, \mathbf{A}_0,u_0) \in \mathcal{C}$ such that $$J(\phi_0,\mathbf{A}_0, u_0)=\min_{(\phi,\mathbf{A},u) \in \mathcal{C}} J(\phi,\mathbf{A},u).$$
\end{thm}
\begin{proof} Let $\{(\phi_n,\mathbf{A}_n,u_n)\}$ be a minimizing sequence (such a sequence exists from the existence result for $u=0$ in \cite{500},
and from the fact that $J$ is lower bounded by $0$).

Thus, such a sequence is such that
$$J(\phi_n, \mathbf{A}_n, u_n) \rightarrow \eta=\inf_{(\phi, \mathbf{A},u) \in \mathcal{C}} J(\phi,\mathbf{A},u).$$

From the expression of $J$, there exists $K>0$ such that
$$\|\nabla \phi_n\|_{0,2,\Omega} \leq K,$$
$$\|\phi_n\|_{0,4,\Omega} \leq K,$$
$$\|\phi_n\|_{0,2,\Omega} \leq K,$$
and
$$\|u_n\|_{0,2,\partial \Omega} \leq K,\; \forall n \in \mathbb{N}$$ so that, from the Rellich-Kondrashov Theorem, there exists
a not relabeled subsequence, $\phi_0 \in W^{1,2}(\Omega,\mathbb{C})$ and $u_0 \in L^2(\Omega,\mathbb{C})$ such that
$$\phi_n \rightharpoonup \phi_0, \text{ weakly in } W^{1,2}(\Omega,\mathbb{C}),$$
 $$\phi_n \rightarrow \phi_0, \text{ in norm  in }  L^2(\Omega,\mathbb{C}) \text{ and }L^4(\Omega,\mathbb{C}),$$
 $$u_n \rightharpoonup u_0, \text{ weakly in } L^2(\partial \Omega,\mathbb{C}), \text{ as } n \rightarrow \infty.$$

 On the other hand,  we have from (\ref{a.2}), from the generalized H\"{o}lder inequality and for  constants $\gamma=\frac{4\pi}{c}\left|\frac{-ie\hbar}{2m}\right|>0$ and $\gamma_1=\frac{4\pi}{c}\frac{2e^2}{mc}>0$ that
 \begin{eqnarray}\label{a.3}0=\rho_{1,n}&\equiv&\langle \text{ curl }\mathbf{A}_n,\text{ curl }\mathbf{A}_n \rangle_{L^2(\Omega_1;\mathbb{R}^3)}
 \nonumber \\ &&-\langle \text{ curl }\mathbf{A}_n,\mathbf{B}_0 \rangle_{L^2(\Omega_1;\mathbb{R}^3)}\nonumber \\ &&+
 \frac{4\pi}{c}\left\langle \frac{ie\hbar}{2m}(\phi^*\nabla \phi-\phi\nabla \phi^*)+\frac{2e^2}{mc}|\phi|^2 \mathbf{A}_n, \mathbf{A}_n\right\rangle_{L^2(\Omega,\mathbb{R}^3)}
 \nonumber \\ &\geq& \langle \text{ curl }\mathbf{A}_n,\text{ curl }\mathbf{A}_n \rangle_{L^2(\Omega_1,\mathbb{R}^3)}
 \nonumber \\ &&- \| \text{ curl }\mathbf{A}_n\|_{0,2,\Omega_1}\|\mathbf{B}_0\|_{0,2,\Omega_1}-\gamma \| \mathbf{A}_n\|_{0,4,\Omega_1}\|\phi_n\|_{0,4,\Omega}\|\nabla \phi_n\|_{0,2,\Omega} \nonumber \\ &&+\gamma_1\left\langle |\phi|^2 ,\mathbf{A}_n\cdot \mathbf{A}_n\right\rangle_{L^2(\Omega,\mathbb{R}^3)}.
 \end{eqnarray}

 From the Friedrichs Inequality (see \cite{780} for details) and the Sobolev Imbedding  Theorem for appropriate constants indicated, we obtain
 \begin{eqnarray}\|\mathbf{A}_n\|_{0,4,\Omega_1}^2 &\leq& K_3 \|\mathbf{A}_n\|_{1,2,\Omega_1}^2 \leq K_4 \left(\|\text{ div } \mathbf{A}_n\|_{0,2,\Omega_1}+
 \|\text{ curl }\mathbf{A}_n\|_{0,2,\Omega}\right)^2 \nonumber \\ &=& K_4\|\text{ curl }\mathbf{A}_n\|_{0,2,\Omega_1}^2,\end{eqnarray}
 since from the London Gauge assumption, $$\text{ div } \mathbf{A}_n=0, \text{ in }\Omega_1,\; \forall n \in \mathbb{N}.$$

Summarizing, we have obtained, for some appropriate $K_5>0$,
\begin{eqnarray}\label{a.4}
0=\rho_{1,n}&\geq& K_5\|\mathbf{A}_n\|_{0,4,\Omega_1}^2+\frac{1}{2}\|\text{ curl } \mathbf{A}_n\|_{0,2,\Omega_1}^2
 \nonumber \\ &&-\| \text{ curl }\mathbf{A}_n\|_{0,2,\Omega_1}\|\mathbf{B}_0\|_{0,2,\Omega_1}-\gamma\|\mathbf{A}_n\|_{0,4,\Omega_1}K^2 \nonumber \\ &&+\gamma_1\left\langle |\phi|^2 ,\mathbf{A}_n\cdot \mathbf{A}_n\right\rangle_{L^2(\Omega;\mathbb{R}^3)}
 \nonumber \\ &\equiv& \rho_{2,n}.
 \end{eqnarray}
Now, suppose to obtain contradiction there exists a subsequence $\{n_k\} \subset \mathbb{N}$ such that $$\|\mathbf{A}_{n_k}\|_{0,4,\Omega_1}
\rightarrow \infty, \text{ as } k \rightarrow \infty.$$ From (\ref{a.4}) we obtain $$\rho_{2,n_k} \rightarrow \infty, \text{ as } k \rightarrow \infty,$$
which contradicts $$\rho_{2,n} \leq 0, \forall n \in \mathbb{N}.$$

Hence, there exists $K_6>0$ such that
$$\|\mathbf{A}_n\|_{0,4,\Omega_1}< K_6,$$
and
$$\|\mathbf{A}_n\|_{0,2,\Omega_1}<K_6,\; \forall n \in \mathbb{N}.$$

From this and (\ref{a.3}) we have,
\begin{eqnarray}\label{a.5}0=\rho_{1,n} &\geq& \|\text{ curl }\mathbf{A}_n \|_{0,2,\Omega_1}^2
 \nonumber \\ &&-\| \text{ curl }\mathbf{A}_n\|_{0,2,\Omega_1}\|\mathbf{B}_0\|_{0,2,\Omega_1}-\gamma K_6 K^2 \nonumber \\ &&+\gamma_1\left\langle |\phi|^2 ,\mathbf{A}_n\cdot \mathbf{A}_n\right\rangle_{L^2(\Omega,\mathbb{R}^3)} \nonumber \\ &\equiv& \rho_{3,n}.
 \end{eqnarray}
 Suppose to obtain contradiction there exists a subsequence $\{n_k\} \subset \mathbb{N}$ such that $$\|\text{ curl } \mathbf{A}_{n_k}\|_{0,2,\Omega_1}
\rightarrow \infty, \text{ as } k \rightarrow \infty.$$ From (\ref{a.5}) we obtain $$\rho_{3,n_k} \rightarrow \infty, \text{ as } k \rightarrow \infty,$$
which contradicts $$\rho_{3,n} \leq 0, \forall n \in \mathbb{N}.$$

Hence, there exists $K_7>0$ such that
$$\|\text{ curl }\mathbf{A}_n\|_{0,2,\Omega_1}< K_7,\; \forall n \in \mathbb{N}$$
so that from this, the Friedrichs inequality and the London Gauge hypothesis, we obtain $K_8>0$ such that
$$\|\mathbf{A}_n\|_{1,2,\Omega_1} <  K_8.$$

So from such a result and the Rellich-Kondrashov Theorem there exists a not relabeled subsequence and $\mathbf{A}_0 \in W^{1,2}(\Omega_1,\mathbb{R}^3)$
such that
$$ \mathbf{A}_n \rightharpoonup \mathbf{A}_0 \text{ weakly } \in W^{1,2}(\Omega_1,\mathbb{R}^3)$$
$$ \mathbf{A}_n \rightarrow \mathbf{A}_0 \text{ in norm  in } L^2(\Omega_1,\mathbb{R}^3) \text{ and } L^4(\Omega_1,\mathbb{R}^3).$$

Moreover from the Sobolev Imbedding Theorem, there exist  real constants $\hat{K}>0, \;\hat{K}_1>0$ such that
$$\|\phi_n\|_{0,6,\Omega} \leq \hat{K}\|\phi_n\|_{1,2,\Omega} < \hat{K}_1, \; \forall n \in \mathbb{N}.$$

Thus, from this and the first equation in (\ref{a.1}), there exist real constants $\hat{K}_2>0,\ldots, \hat{K}_6>0,$ such that
\begin{eqnarray}
\|\nabla^2 \phi_n\|_{0,2,\Omega} &<& \hat{K}_2\|\mathbf{A}_n\|_{0,2,\Omega}\|\nabla \phi_n\|_{0,2,\Omega} \nonumber \\
&& \hat{K}_3\|\mathbf{A}_n\|_{0,4,\Omega}\|\phi_n\|_{0,2,\Omega}+\hat{K}_4\|\phi_n\|_{0,6,\Omega}^3+\hat{K}_5\|\phi_n\|_{0,2,\Omega}
\nonumber \\  &\leq& \hat{K}_6,\; \forall n \in \mathbb{N}.\end{eqnarray}

From this, up to a subsequence, we get
$$\nabla^2 \phi_n \rightharpoonup \nabla^2 \phi_0 \text{ weakly in } L^2(\Omega;\mathbb{C}).$$

Let $$\varphi \in C_c^\infty(\Omega,\mathbb{C}),\; \varphi_1 \in C_c^\infty(\Omega,\mathbb{R}^3) \text{ and } \varphi_2 \in C_c^\infty(\Omega_1\setminus \Omega,\mathbb{R}^3).$$

From the last results, we may easily obtain the following limits

\begin{enumerate}
\item $$\langle \nabla^2 \phi_n, \varphi\rangle_{L^2} \rightarrow \langle \nabla^2 \phi_0, \varphi\rangle_{L^2},$$
\item $$\langle \nabla \phi_n, \nabla \varphi \rangle_{L^2} \rightarrow \langle \nabla \phi_0, \nabla \varphi \rangle_{L^2},$$
\item $$\langle \mathbf{A}_n \cdot \nabla\phi_n,  \varphi \rangle_{L^2} \rightarrow \langle \mathbf{A}_0 \cdot \nabla \phi_0, \varphi \rangle_{L^2},$$
\item\label{a.9} $$\langle |\mathbf{A}_n|^2 \phi_n, \varphi \rangle_{L^2} \rightarrow \langle  |\mathbf{A}_0|^2 \phi_0,  \varphi \rangle_{L^2},$$
\item $$\langle |\phi_n|^2 \phi_n, \varphi \rangle_{L^2} \rightarrow \langle  |\phi_0|^2 \phi_0,  \varphi \rangle_{L^2},$$
\item $$\langle \text{ curl }\mathbf{A}_n, \text{ curl }\varphi_1 \rangle_{L^2} \rightarrow \langle \text{ curl }\mathbf{A}_0, \text{ curl }\varphi_1 \rangle_{L^2},$$
\item $$\langle \phi_n^* \nabla \phi_n, \varphi_1\rangle_{L^2} \rightarrow \langle \phi_0^* \nabla \phi_0, \varphi_1\rangle_{L^2},$$
\item $$\langle \phi_n \nabla \phi_n^*, \varphi_1\rangle_{L^2} \rightarrow \langle \phi_0 \nabla \phi_0^*, \varphi_1\rangle_{L^2},$$
\item $$\langle |\phi_n|^2 \mathbf{A}_n, \varphi_1\rangle_{L^2} \rightarrow \langle |\phi_0|^2 \mathbf{A}_0, \varphi_1\rangle_{L^2}.$$
\end{enumerate}

For example, for (\ref{a.9}), for an appropriate real $\tilde{K}>0$ we have
\begin{eqnarray}
&&|\langle |\mathbf{A}_n|^2 \phi_n,\varphi\rangle_{L^2}-\langle |\mathbf{A}_0|^2 \phi_0,\varphi\rangle_{L^2}|
\nonumber \\ &=&|\langle |\mathbf{A}_n|^2 \phi_n,\varphi\rangle_{L^2}-\langle |\mathbf{A}_0|^2 \phi_n,\varphi\rangle_{L^2}
+\langle |\mathbf{A}_0|^2 \phi_n,\varphi\rangle_{L^2}-\langle |\mathbf{A}_0|^2 \phi_0,\varphi\rangle_{L^2}|
\nonumber \\ &\leq&
|\langle |(\mathbf{A}_n|^2-|\mathbf{A}_0|^2) \phi_n,\varphi\rangle_{L^2}+\langle |\mathbf{A}_0|^2(\phi_n- \phi_0),\varphi\rangle_{L^2}|
\nonumber \\ &\leq&
|\langle |(\mathbf{A}_n|-|\mathbf{A}_0|)(|\mathbf{A}_n|+|\mathbf{A}_0|) \phi_n,\varphi\rangle_{L^2}+\langle |\mathbf{A}_0|^2(\phi_n- \phi_0),\varphi\rangle_{L^2}|
\nonumber \\ &\leq&
\|(\mathbf{A}_n|+|\mathbf{A}_0|)\|_{0,4,\Omega}\||\mathbf{A}_n|-|\mathbf{A}_0|\|_{0,4,\Omega}\|\phi_n|_{0,2,\Omega} \|\varphi\|_\infty+\||\mathbf{A}_0|^2\|_{0,2,\Omega}\|\phi_n- \phi_0\|_{0,2,\Omega}\|\varphi\|_\infty
\nonumber \\ &\leq& \tilde{K}(\||\mathbf{A}_n|-|\mathbf{A}_0|\|_{0,4,\Omega}+\|\phi_n- \phi_0\|_{0,2,\Omega}) \nonumber \\ &\rightarrow& 0, \text{ as } n \rightarrow \infty.
\end{eqnarray}

The other items may be proven similarly.

Now let $\varphi \in C^\infty(\overline{\Omega},\mathbb{C}).$
Observe that

\begin{eqnarray}
&&\langle u_n,\varphi\rangle_{L^2(\partial \Omega,\mathbb{C})}
\nonumber \\ &=& \left\langle \left(i\hbar \nabla \phi_n+\frac{2e}{c}\mathbf{A}_n\phi_n\right)\cdot \mathbf{n}, \varphi\right\rangle_{L^2(\partial \Omega,\mathbb{C})}
\nonumber \\ &=& \left\langle i\hbar \nabla \phi_n+\frac{2e}{c}\mathbf{A}_n\phi_n, \nabla \varphi\right\rangle_{L^2(\Omega,\mathbb{C}^3)}
\nonumber \\ && +\left\langle \text{div }\left( i\hbar \nabla \phi_n+\frac{2e}{c}\mathbf{A}_n\phi_n\right),  \varphi\right\rangle_{L^2(\Omega,\mathbb{C})}
\nonumber \\ &\rightarrow& \left \langle i\hbar \nabla \phi_0+\frac{2e}{c}\mathbf{A}_0\phi_0, \nabla \varphi\right\rangle_{L^2(\Omega,\mathbb{C}^3)}
\nonumber \\ && +\left\langle \text{div }\left( i\hbar \nabla \phi_0+\frac{2e}{c}\mathbf{A}_0\phi_0\right),  \varphi\right\rangle_{L^2(\Omega,\mathbb{C})}
\nonumber \\ &=& \left\langle \left(i\hbar \nabla \phi_0+\frac{2e}{c}\mathbf{A}_0\phi_0\right)\cdot \mathbf{n}, \varphi\right\rangle_{L^2(\partial \Omega,\mathbb{C})}.
\end{eqnarray}

From this and from $$\langle u_n,\varphi\rangle_{L^2(\partial \Omega,\mathbb{C})} \rightarrow \langle u_0,\varphi\rangle_{L^2(\partial \Omega,\mathbb{C})},$$
we have
$$\left\langle \left(i\hbar \nabla \phi_0+\frac{2e}{c}\mathbf{A}_0\phi_0\right)\cdot \mathbf{n}-u_0, \varphi\right\rangle_{L^2(\partial \Omega,\mathbb{C})}=0,\; \forall \varphi \in C^\infty(\overline{\Omega},\mathbb{C}),$$
so that in such a distributional sense,
$$\left(i\hbar \nabla \phi_0+\frac{2e}{c}\mathbf{A}_0\phi_0\right)\cdot \mathbf{n}=u_0, \text{ on } \partial \Omega.$$

The other boundary condition may be dealt similarly. Thus, from these last results we may infer that in the distributional sense,
\begin{equation}\left\{
\begin{array}{ll}
 \frac{1}{2m}\left(-i\hbar \nabla -\frac{2e}{c}\mathbf{A}_0 \right)^2\phi_0+\alpha |\phi_0|^2\phi_0 -\beta \phi_0 =0,& \text{ in } \Omega,
  \\ \\
  \left(i\hbar \nabla \phi_0+ \frac{2e}{c}\mathbf{A}_0\phi_0\right) \cdot \mathbf{n}=u_0,& \text{ on } \partial  \Omega,\end{array} \right.\end{equation}

and

\begin{equation}\left\{
\begin{array}{ll}
 \text{ curl } \text{ curl } \mathbf{A}_0 =\text{ curl }\mathbf{B}_0+\frac{4\pi}{c} \tilde{J}_0,& \text{ in } \Omega,
  \\ \\
  \text{ curl } \text{ curl } \mathbf{A}_0=\text{ curl } \mathbf{B}_0,& \text{ in } \Omega_1\setminus \Omega,
  \\ \\ \text{ div }\mathbf{A}_0=0,& \text{ in } \Omega_1,
  \\ \\
  \mathbf{A}_0 \cdot \mathbf{n}=0,& \text{ on } \partial \Omega_1
  \end{array} \right.\end{equation}
  where,
 $$\tilde{J}_0=-\frac{ie\hbar}{2m}(\phi^*_0\nabla \phi_0-\phi_0\nabla \phi^*_0)-\frac{2e^2}{mc}|\phi_0|^2 \mathbf{A}_0.$$

 Hence $(\phi_0,\mathbf{A}_0,u_0) \in \mathcal{C}.$

 Finally, from $\phi_n \rightarrow \phi_0$ in $L^2$ and $L^4$,  $\phi_n \rightharpoonup \phi_0 \text{ weakly  in } W^{1,2},$
 $u_n \rightharpoonup u_0 \text{ weakly in } L^2(\partial \Omega),$ by continuity in $\phi$ and the convexity of $J$ in $\nabla \phi$
 and $u$, we have,

 $$\eta=\liminf_{n \rightarrow \infty} J(\phi_n,\mathbf{A}_n, u_n) \geq J(\phi_0,\mathbf{A}_0,u_0).$$

The proof is complete.
\end{proof}
\section{Conclusion}

In this article we have developed a global existence result for a control problem related to the Ginzburg-Landau system in superconductivity.
We emphasize the control variable $u$ acts on the super-conducting sample boundary, whereas the state variables, namely, the order parameter
$\phi$ and the magnetic potential $\mathbf{A}$ are defined on $\Omega$ and $\Omega_1$, respectively. The problem has non-linear constraints but the cost functional is convex. Finally, we highlight the London Gauge assumption and the Friedrichs Inequality have a fundamental role in the establishment of the main results.

\end{document}